\documentclass[a4paper,12pt]{article}

\usepackage{amssymb, amsmath, amsthm}
\usepackage{hyperref}

\title{On a Conjecture on the Representation of Positive Integers as the Sum of Three Terms of the Sequence $ \floor{\frac{n^2}{a}}$}

\author{ 
        \textbf{Sebastian Tim Holdum} \\ 
        {\small Niels Bohr Institute, University of Copenhagen, Denmark}\\
        \texttt{sebastian.holdum@nbi.dk} \vspace{10pt} \\
        \textbf{Frederik Ravn Klausen} \\
        {\small Department of Mathematics, University of Copenhagen, Denmark}\\ 
        \texttt{tlk870@alumni.ku.dk} \vspace{10pt} \\  
        \textbf{Peter Michael Reichstein Rasmussen}  \\
        {\small Department of Mathematics, University of Copenhagen, Denmark}\\ 
        \texttt{nmq584@alumni.ku.dk}
       }
       
\date{}

\newcommand{\floor}[1] {
\left\lfloor #1 \right\rfloor
}

\newcommand{\fpart}[1] {
\left\langle #1 \right\rangle
}
 
\newcounter{thmcounter}

\newtheorem{theorem}[thmcounter]{Theorem}
\newtheorem{corollary}[thmcounter]{Corollary}
\newtheorem{conjecture}[thmcounter]{Conjecture}
\newtheorem{observation}[thmcounter]{Observation}
\theoremstyle{definition}
\newtheorem{definition}[thmcounter]{Definition}

\begin{document}
\maketitle

\begin{abstract}
We prove some cases of a conjecture by Farhi on the representation of every positive integer as the sum of three terms of the sequence $ \floor{\frac{n^2}{a}}$. This is done by generalizing a method used by Farhi in his original paper.
\end{abstract}

\section{Introduction}
In the following we let $\mathbb{N}$ denote the set of non-negative integers, $\floor{\cdot}$ is the integer part function, and $\fpart{\cdot}$ is the fractional part function.

A classical result by Legendre \cite{threeSquare} states that every natural number not of the form $4^s(8t+7), s, t\in \mathbb{N}$ can be written as the sum of three squares.

 In relation to this Farhi recently conjectured the following:

\begin{conjecture}[Farhi \cite{Farhi2}]\label{main}
Let $a\geq3$ be an integer. Then every natural number can be represented as the sum of three terms of the sequence $\left( \floor{\frac{n^2}{a}} \right)_{n\in \mathbb{N}}$.
\end{conjecture}

The conjecture was confirmed by Farhi \cite{Farhi1} and Mezroui, Azizi, and Ziane \cite{analyt} for $a\in\{3, 4, 8\}$.

In this paper we generalize the method used by Farhi for $a=4$, and partially for $a=3$, to prove that the conjecture holds for $a \in \{$4, 7, 8, 9, 20, 24, 40, 104, 120$\}$. The method uses Legendre's three-square theorem and properties of quadratic residues.

We also note that the set of integers, $a$, such that Conjecture \ref{main} holds is closed under multiplication by a square.

\section{Method and results}
We start by introducing the following sets:
\begin{definition}
For any nonzero $a\in \mathbb N$ we define
\begin{align*}
\mathcal{Q}_a=\{0<\varphi<a \mid \exists x \in \mathbb{Z} \colon \varphi\equiv x^2\pmod a \}.
\end{align*}
Therefore, $\mathcal{Q}_a$ is the set of quadratic residues modulo $a$.
\end{definition}

\begin{definition}
	For any nonzero $a \in \mathbb{N}$ we define
	\begin{align*}
	\mathcal{A}_a=\{\varphi\in\mathbb{N} \mid \exists x,y,z \in \mathcal{Q}_a \cup \{ 0\} \colon \varphi=x+y+z\}.
	\end{align*}
	Thus, $\mathcal{A}_a$ is the set of integers that can be written as the sum of three elements of $\mathcal{Q}_a\cup \{0\}$.
\end{definition}

\begin{definition}
	For any nonzero $a\in \mathbb N$ we define
	\begin{align*}
	\mathcal{R}_a=\{\varphi\in \mathcal{A}_a \mid \forall \psi \in \mathcal{A}_a \colon \varphi\equiv \psi \pmod{a} \Rightarrow \varphi = \psi \}.
	\end{align*}
	So, $\mathcal{R}_a$ is the set of integers that can be written as the sum of three elements of $\mathcal{Q}_a\cup \{0\}$, and such that no other integer in the same residue class modulo $a$ has this property.
\end{definition}

Now, we are ready to formulate the main result.
\begin{theorem}\label{hovedresultat}
Let $a\in \mathbb{N}$ be nonzero and assume that for every $k\in \mathbb{N}$ there exists an $r\in \mathcal{R}_a$ such that $ak+r \neq 4^s(8t+7)$ for any $s, t \in \mathbb{N}$. Then every $N\in \mathbb{N}$ can be written as the sum of three terms of the sequence $\left( \floor{\frac{n^2}{a}} \right)_{n\in \mathbb{N}}$.
\end{theorem}
\begin{proof}
Let $N\in \mathbb{N}$ be fixed. By assumption we can choose $r \in \mathcal{R}_a$ such that $aN+r \neq 4^s(8t+7)$ for any $s, t \in \mathbb{N}$. By Legendre's theorem it follows that $aN+r$ can be written of the form
\begin{equation}\label{eq}
  aN+r = A^2+B^2+C^2
\end{equation}
for some $A, B, C\in \mathbb{N}$. Now we have
\begin{equation*}
  r\equiv A^2+B^2+C^2 \pmod{a},
\end{equation*}
so
\begin{equation}\label{red}
r = (A^2 \bmod a) + (B^2 \bmod a) + (C^2 \bmod a),
\end{equation}
since $r\in \mathcal{R}_a$. 
Dividing by $a$ and separating the integer and fractional parts of the right hand side in \eqref{eq}, we get
\begin{equation*}
  N+\frac{r}{a} = \floor{\frac{A^2}{a}}+\floor{\frac{B^2}{a}}+\floor{\frac{C^2}{a}}+\fpart{\frac{A^2}{a}}+\fpart{\frac{B^2}{a}}+\fpart{\frac{C^2}{a}},
\end{equation*}
and from \eqref{red} we have
\begin{align*} 
\frac{r}{a} = \fpart{\frac{A^2}{a}}+\fpart{\frac{B^2}{a}}+\fpart{\frac{C^2}{a}},
\end{align*}
so
\begin{align*}
N = \floor{\frac{A^2}{a}}+\floor{\frac{B^2}{a}}+\floor{\frac{C^2}{a}}.
\end{align*}
\end{proof}

 Since we can find the sets $\mathcal{R}_a$ by computation, we can now apply the main theorem to get the following corollary.
\begin{corollary} \label{vals}
Conjecture \ref{main} is satisfied for $a\in \{4, 7, 8, 9, 20, 24, 40, 104, 120\}$.
\end{corollary}
\begin{proof}
Consider the following table:
\begin{center}
	\begin{tabular}{c|l}
		$a$ & $\mathcal{R}_a$ \\ \hline
		4 & $\{0, 1, 2, 3\}$ \\ 
		7 & $\{4, 6\}$ \\ 
		8 & $\{2, 3, 5, 6\}$ \\ 
		9 & $\{1, 4, 7, 8\}$ \\ 
		20 & $\{11, 15, 18, 19\}$ \\ 
		24 & $\{11, 14, 19, 21, 22\}$ \\ 
		40 & $\{27, 38\}$ \\ 
		104 & $\{99\}$ \\ 
		120 & $\{107\}$
	\end{tabular}
\end{center}

Calculating modulo 8 it can be checked fairly easily that for each $a\in \{$4, 7, 8, 9, 20, 24, 40, 104, 120$\}$ and every $k\in \mathbb{N}$ there exists an $r\in \mathcal{R}_a$ such that $ak+r$ is not of the form $4^s(8t+7), s, t\in \mathbb{N}$, and thus every natural number can be written as the sum of three terms of the sequence $\left( \floor{\frac{n^2}{a}} \right)_{n\in \mathbb{N}}$.

To demonstrate this, we show the case $a=7$. All the other cases are done in exactly the same way.

For $k\equiv 1, 2, 3, 6\text{ or } 7$ (mod 8) we have $7k+4 \equiv 3, 2, 1, 6 \text{ and } 5$ (mod 8), respectively, and for $k\equiv 0, 4 \text{ or } 5$ (mod 8) we have $7k+6 \equiv 6, 2 \text{ and } 1$ (mod 8), respectively. Since $4^s(8t+7) \equiv 0, 4 \text{ or } 7$ (mod 8), $s, t\in \mathbb{N}$, we conclude that for every $k\in \mathbb{N}$ we can write $7k+r$, for $r\in \mathcal{R}_7=\{4, 6\}$, such that it is not of the form $4^s(8t+7), s, t\in \mathbb{N}$. The case now follows from Theorem \ref{hovedresultat}. 
\end{proof}

Further, one should note that the set of integers satisfying Conjecture \ref{main} is closed under multiplication by a square.

\begin{observation}\label{oplagt}
    Let $\mathcal{M}$ be the set of integers satisfying Conjecture \ref{main}. If $a\in \mathcal{M}$, then $ak^2\in \mathcal{M}$ for any integer $k>0$.
\end{observation}
\begin{proof}
	This follows easily since for any $n$ we can find $A, B, C\in \mathbb{N}$ such that
	\begin{align*}
        n &= \floor{\frac{A^2}{a}} + \floor{\frac{B^2}{a}} + \floor{\frac{C^2}{a}}\\
          &= \floor{\frac{(Ak)^2}{ak^2}}+ \floor{\frac{(Bk)^2}{ak^2}} + \floor{\frac{(Ck)^2}{ak^2}}.
    \end{align*}
\end{proof}
Knowing this, we see that since Conjecture \ref{main} is satisfied for $a= 3, 9, 4$, and $8$, it must also hold for $a=3^k$ for any positive integer $k$ and for $a=2^k, k>1$. 

Finally, using Observation \ref{oplagt}, Corollary \ref{vals}, and the fact \cite{analyt} that Conjecture \ref{main} holds for $a=3$, we get that the conjecture holds for the following values up to 120.
\begin{align*}
  a\in \{&3, 4, 7, 8, 9, 12, 16, 20, 24, 27, 28, 32, 36, 40, 48, \\ &63, 64, 72, 75, 80, 81, 96, 100, 104, 108, 112, 120\}.
\end{align*}
Unfortunately, it seems that the method deployed in Theorem \ref{hovedresultat} is not extendable to other cases, since its success relies on $\mathcal{R}_a$, and in general $\mathcal{R}_a$ does not contain the necessary elements for the condition in the theorem to be satisfied.

\section{Acknowledgement}
The authors would like to thank Jan Agentoft Nielsen for his suggestions that helped to improve the manuscript.

\textbf{2010 Mathematics Subject Classification:} Primary 11B13.\\
\textbf{Keywords:} additive base, Legendre's theorem.


\begin{thebibliography}{9}
\bibliographystyle{plain}

\bibitem{Farhi1}B. Farhi, 
On the representation of the natural numbers as the sum of three terms of the sequence $\floor{\frac{n^2}{a}}$, {\it J. Integer Seq.,} {\bf
16} (2013),
\href{https://cs.uwaterloo.ca/journals/JIS/VOL16/Farhi/farhi7.html}{Article
	13.6.4}.

\bibitem{Farhi2} B. Farhi, 
An elementary proof that any natural number can be written as the sum of three terms of the sequence $\floor{\frac{n^2}{3}}$, {\it J. Integer Seq.,} {\bf
17} (2014),
\href{https://cs.uwaterloo.ca/journals/JIS/VOL17/Farhi/farhi12.html}{Article
	14.7.6}.

\bibitem{threeSquare}
A. M. Legendre, {\it Th\'eorie des Nombres}, 3rd ed., Vol.\ 2, 1830.

\bibitem{analyt} S. Mezroui, A. Azizi, and M. Ziane, 
On a conjecture of Farhi, {\it J. Integer Seq.}, {\bf 17} (2014),
\href{https://cs.uwaterloo.ca/journals/JIS/VOL17/Mezroui/soufiane4.html}{Article
	14.1.8}.
\end{thebibliography}
\end{document}